\documentclass[11pt]{amsart}

\usepackage[english]{babel}

\usepackage{amsmath}
\usepackage{graphicx}
\usepackage{complexity}
\usepackage{xspace}
\usepackage{tikz}
\usetikzlibrary{calc}
\usepackage{amssymb}     
\usepackage{amsthm}  
\usepackage{bbm}    
\usepackage{bm}
\usepackage{accents}
\usepackage{ulem}
\usepackage{mathrsfs}
\usepackage{fullpage}
\usepackage{microtype}
\usepackage[colorlinks,linkcolor=blue,citecolor=magenta,pagebackref]{hyperref}
\usepackage[colorinlistoftodos,bordercolor=orange,backgroundcolor=orange!20,linecolor=orange,textsize=scriptsize]{todonotes}
\usepackage{soul}
\usepackage{yhmath}
\usepackage{subfig}
\usepackage{mathtools}
\usepackage[sharp]{easylist}
\usepackage[inline]{enumitem}
\setlist[enumerate]{leftmargin=.5in, rightmargin=.5in}

\renewcommand{\geq}{\geqslant}
\renewcommand{\leq}{\leqslant}
\renewcommand{\preceq}{\preccurlyeq}

\def\calP{\mathcal{P}}
\def\calA{\mathcal{A}}
\def\calK{\mathcal{K}}
\def\calI{\mathcal{I}}
\def\calD{\mathcal{D}}

\newtheorem{theorem}{Theorem}[section]
\newtheorem{proposition}[theorem]{Proposition}

\newtheorem{lem}[theorem]{Lemma}

\newtheorem*{thm-nonumb}{Theorem}

\theoremstyle{remark}

\theoremstyle{definition}

\newcommand{\tor}{\overset{r}{\to}}
\newcommand{\tob}{\overset{b}{\to}}

\title{Revisiting classical results on kernels in digraphs}

\author{Hélène Langlois \and Frédéric Meunier}

\address[Hélène Langlois]{LAMA, Université Paris-Est Créteil, Créteil, France}
\email{helene.langlois@u-pec.fr}

\address[Frédéric Meunier]{CERMICS, ENPC, Institut Polytechnique de Paris, Marne-la-Vallée, France}
\email{frederic.meunier@enpc.fr}

\begin{document}

\begin{abstract}
In a digraph, a kernel is a subset of vertices that is both independent and absorbing. Kernels have important applications in combinatorics and outside. Kernels do not always exist and finding sufficient conditions ensuring their existence is a key theoretical challenge. In this work, we revisit and generalize a few classical results of this sort, especially the Sands--Sauer--Woodrow theorem and the Galeana-Sánchez--Neumann-Lara theorem.
\end{abstract}


\maketitle

\normalem

\section{Introduction}

In a digraph, a subset of vertices is \emph{independent} if it does not contain adjacent vertices, and it is \emph{absorbing} if every vertex is either in the subset or has an outneighbor in it.  A \emph{kernel} is a subset of vertices that is both independent and absorbing. Formally, a kernel of a digraph $D$ is a subset $S$ of $V(D)$ that is independent and such that $V(D)\setminus S = N^-(S)$. Kernels form a fundamental topic of the theory of digraphs. Introduced by von Neumann and Morgenstern in the context of board games analysis~\cite{vonneumann1947}, they have found applications in other areas like  economy~\cite{igarashi_coalition_2017} and logic~\cite{walicki2017resolving}. They have been the subject of many research works and there are still several challenges about them. There are graphs with no kernel, e.g., the directed cycle of length three, and actually it is even $\NP$-complete to decide whether a digraph admits a kernel~\cite{chvatal_computational_1973}.

This paper aims at revisiting three classical results stating sufficient conditions for the existence of kernels, and at providing generalizations of them. In this paper, all directed graphs are simple : there is no loop and no parallel arcs, i.e. no two arcs with the same tail and the same head. (Note however that ``opposite'' arcs---the tail of one arc being the head of the other, and vice-versa---are allowed.)

\subsection{Digraphs with red and blue arcs} One of the most celebrated theorems about kernels is the following. It is actually a generalization of the Gale--Shapley ``marriage'' theorem~\cite{gale_college_1962}.

\begin{thm-nonumb}[Sands, Sauer, and Woodrow~\cite{sands1982monochromatic}]
    Let $D$ be a digraph whose arcs are colored with two colors. Then there is a subset $S$ of vertices such that no two vertices in $S$ are connected by a monochromatic directed path and such that from every vertex there is a monochromatic directed path ending in $S$.
\end{thm-nonumb}

 An equivalent formulation of this theorem is the following: \emph{Let $D$ be a digraph whose arcs are colored with two colors such that the restriction to each color forms a transitive digraph; then $D$ admits a kernel}. It is moreover easy and well-known that computing a kernel under the condition of the theorem can be done in polynomial time. See Fleiner~\cite{fleiner2002stable} for discussions around this result.

Our first contribution shows that it is possible to extend the existence result to a dramatically larger family of digraphs with red and blues arcs, while keeping polynomial-time computability. This family is obtained by replacing the condition of the Sands--Sauer--Woodrow theorem by the one depicted in Figure~\ref{fig:br}. The original condition is obtained by restricting each implication to its first alternative. Our result is also a generalization of a theorem by Champetier~\cite{champetier1989kernels} ensuring that so-called ``$M$-clique-acyclic orientations'' of comparability graphs always admit kernels, and of a theorem by Abbas and Saoula~\cite{abbas2005polynomial} that states the polynomiality of computing such a kernel under the same condition. $M$-clique-acyclic orientations have played an important role in the theory of kernels; see Section~\ref{subsec:terminology} for their definition. 
\begin{figure}
\centering
\includegraphics[]{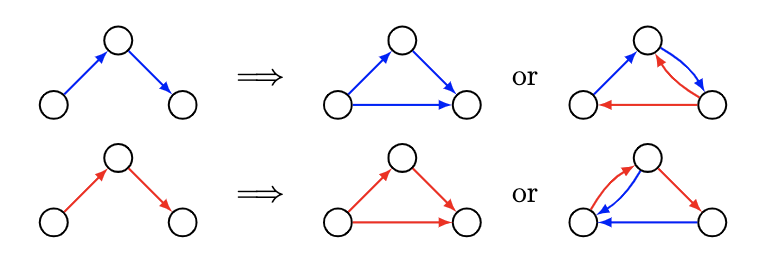}

\caption{\label{fig:br} The condition in Theorem~\ref{thm:br}}
\end{figure}

For two vertices $u,v$, we write $u\tob v$ (resp.\ $u\tor v$) to express the existence of a blue (resp.\ red) arc from $u$ to $v$. 

\begin{theorem}\label{thm:br}
    Let $D$ be a digraph whose arcs are colored in blue and red, such that the following conditions are both satisfied (see Figure~\ref{fig:br}):
    \begin{enumerate}[label=\textup{(\roman*)}]
        \item\label{item:b} If $u \tob v$ and $v \tob w$, then $u \tob w$ or ($w \tor u$ and $w \tor v$).
        \item\label{item:r} If $u \tor v$ and $v \tor w$, then $u \tor w$ or ($v \tob u$ and $w \tob u$).
    \end{enumerate}
    Then $D$ has a kernel, and it is possible to compute such a kernel in polynomial time. 
\end{theorem}

This theorem is proved in Section~\ref{sec:red-blue}, where related results are also discussed. We emphasize that the conditions~\ref{item:b} and~\ref{item:r} are not symmetric. We do not know whether a more general statement with symmetric conditions is possible.

\subsection{Chords in odd holes}

A celebrated theorem about kernels is due to Richardson~\cite{richardson_weakly_1946}. It states that every digraph with no odd directed cycle has a kernel. The next theorem is a far-reaching generalization. A {\em chord} of a directed cycle is an arc whose endpoints are on the cycle but that does not belong to the cycle.

\begin{thm-nonumb}[Galeana-Sánchez and Neumann-Lara~\cite{galeana1984kernels}]\label{thm_gsnl}
Let $D$ be a digraph. Suppose that in $D$ each directed cycle of odd length has at least two chords with consecutive heads. Then $D$ has a kernel.
\end{thm-nonumb}

We extend this theorem by relaxing the condition on the odd cycles. The proof is given in Section~\ref{sec:odd-holes}. Two chords $(u,v)$ and $(w,t)$ are {\em crossing} if $u,w,v,t$ are distinct and come  either in that order around the cycle, or in the opposite order. They are {\em nested} if $u,w,t,v$ are distinct and come  either in that order, or in the opposite order. A chord is {\em odd} if the part of the cycle from its tail to its head is of odd length. It is {\em short} is this part is of length two.

\begin{theorem}\label{thm:cordes_nous}
    Let $D$ be a digraph. Suppose that in $D$ each odd directed cycle has
    \begin{itemize}
        \item two chords with consecutive heads, or
        \item two odd chords that are neither crossing nor nested (see Figure~\ref{fig:cordes_nous_same_dir}), or
        \item two crossing chords, one being short and the other being odd (see Figure~\ref{fig:cordes_nous_short_odd}).
    \end{itemize}
    Then $D$ has a kernel.
\end{theorem}

We emphasize that directed cycles of length three under the condition of the theorem have necessarily two ``reversible'' arcs, i.e., there are two arcs $(u,v)$ and $(v,w)$ of the cycles such that the arcs $(v,u)$ and $(w,v)$ exist in the graph as well. In other words, directed graphs satisfying the condition of the theorem are ``$M$-clique-acyclic;'' see Section~\ref{sec:kernel-solvability} for a discussion of this fundamental notion in the area of kernels.

\begin{figure}
    \centering
\includegraphics[]{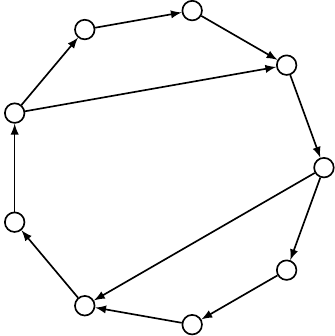}
    \caption{An odd directed cycle having two odd chords that are neither crossing nor nested}
    \label{fig:cordes_nous_same_dir}
\end{figure}

\begin{figure}
    \centering
\includegraphics[]{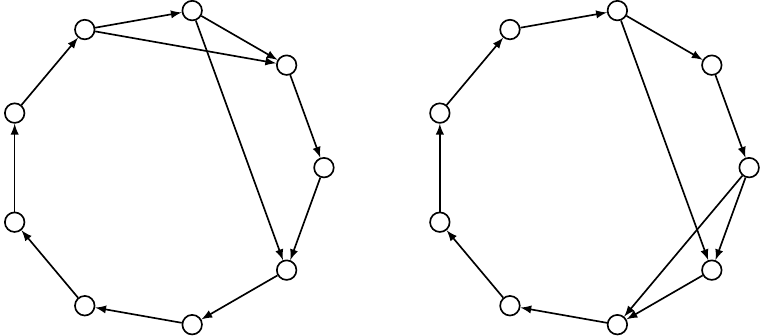}
    \caption{Odd directed cycles having two crossing chords, one short and the other being odd}
    \label{fig:cordes_nous_short_odd}
\end{figure}

There is no algorithmic versions of the Galeana-Sánchez--Neumann-Lara theorem in the literature. Note that the complexity status of checking that a graph satisfies the condition of the theorem is unclear. The same situation holds for our extension. 

\subsection{Clique-acyclic orientations of anti-holes}\label{subsec:clique-acyclic}

A digraph is \emph{clique-acyclic} if every clique has a vertex whose inneighborhood contains all other vertices of the clique (i.e., the clique induces a subgraph with a kernel). An undirected graph is \emph{kernel-solvable} if every clique-acyclic orientation admits a kernel. This is a fundamental notion about kernels, introduced by Berge and Duchet~\cite{berge_probleme_1983} in relation with the strong graph perfect conjecture. They even conjectured that every perfect graph is kernel-solvable, which has been proved to be true by Boros and Gurvich~\cite{boros_perfect_1996}.

We identify and partially fix an inaccuracy in the literature regarding kernel-solvability. To do so, we introduce a new notion, \emph{simple kernel-solvability}, defined as follows: an undirected graph is \emph{simple kernel-solvable} if every simple clique-acyclic orientation admits a kernel. (We remind the reader that in a simple orientation no edge is reversible, i.e., each edge is oriented in a single way.)

In Section~\ref{sec:kernel-solvability}, we prove the following.

\begin{proposition}\label{prop:odd_anti-hole}
    Every odd anti-hole $\overline C_n$ with $n \geq 9$ is simple kernel-solvable but not kernel-solvable.
\end{proposition}

In the literature, it has even been claimed that $\overline C_7$ is 
kernel-$M$-solvable (see below for the definition), which would imply that kernel-$M$-solvability does not characterize perfect graphs. The argument of the literature is actually not correct since kernel-$M$-solvability implies simple kernel-solvability, and $\overline C_7$ is not even simple kernel-solvable. This leaves open the exact relation between perfect graphs and kernel-$M$-solvability. See Section~\ref{sec:kernel-solvability} for a detailed discussion.


\subsection{Some terminology and notation}\label{subsec:terminology} Given a digraph and two vertices $u,v$, we write $u \to v$ to express the existence of an arc $(u,v)$. This notation is extended to subsets of vertices: given a vertex $u$ and a subset $U$ of vertices, $u \to U$ (resp.\ $U \to u$) means the existence of at least one arc from $u$ to some vertex of $U$ (resp.\ from some vertex of $U$ to $u$). In case where $u \in U$, this means the existence of arc to another vertex (resp.\ from another vertex) distinct from $u$.

An arc $(u,v)$ is {\em reversible} if $(v,u)$ is also an arc. A digraph is {\em transitive} if $u\to v$ and $v \to w$ imply $u \to w$. The set of outneighbors (resp.\ inneighbors) of a vertex $v$ is denoted by $N^+(v)$ (resp.\ $N^-(v)$). The set $N^+(v) \cup \{v\}$ (resp.\ $N^-(v)\cup\{v\}$) is denoted by $N^+[v]$ (resp.\ $N^-[v]$). Given a subset $U$ of vertices, $N^+(U)$ (resp.\ $N^-(U)$) denotes $\bigl(\bigcup_{v \in U}N^+(v)\bigl) \setminus U$ (resp.\ $\bigl(\bigcup_{v \in U}N^-(v)\bigl) \setminus U$). 

An orientation of a graph is {\em $M$-clique-acyclic} if every directed cycle of length three has at least two reversible arcs. A undirected graph is \emph{kernel-$M$-solvable} if every $M$-clique-acyclic orientation has a kernel.


\subsection*{Acknowledgments}
This research was partially supported by Labex Bézout ``Models and Algorithms: From the Discrete to the Continuous.''

\section{``Red--blue'' arcs}\label{sec:red-blue}

\subsection{Proof of Theorem~\ref{thm:br}}

\subsubsection{The poset of antichains}\label{subsec:poset}
The following way of extending a partial order on antichains will be useful in the proof of Theorem~\ref{thm:br}. Let $(\calP,\preceq)$ be a poset. Let $\calA$ be the collection of all antichains of this poset. We extend $\preceq$ on $\calA$ by setting $\alpha\preceq\alpha'$ for two antichains $\alpha,\alpha'\in\calA$ whenever for each element $x$ in $\alpha$, there exists an element $x'$ in $\alpha'$ such that $x \preceq x'$. We believe this construction and the following lemma are well-known but we have not been able to find any reference in the literature. 

\begin{lem}\label{lem:apos}
    With this extended definition, $\preceq$ is a partial order on $\calA$.
\end{lem}

\begin{proof}
    Reflexivity and transitivity are immediate. We establish antisymmetry. Suppose for a contradiction that there exist two antichains $\alpha \neq \alpha'$ such that $\alpha\preceq\alpha'$ and $\alpha'\preceq\alpha$. Since $\alpha \neq \alpha'$, we can assume without loss of generality that there exists $x \in \alpha$ such that $x \notin \alpha'$. Since $\alpha \preceq \alpha'$, there exists $x' \in \alpha'$ such that $x \prec x'$. Since $\alpha' \preceq \alpha$, there exists $x''\in\alpha$ such that $x' \preceq x''$, and thus such that $x \prec x''$; a contradiction.
\end{proof}

The next lemma will be useful to establish the polynomiality of the computation of the kernel. Actually, only the upper bound on the size of a chain is needed for the proof but we establish the existence of a chain matching the upper bound for sake of completeness.

\begin{lem}\label{lem:hA}
    If $\calP$ is finite, then the maximal size of a chain in $(\calA,\preceq)$ is $|\calP|+1$.
\end{lem}

\begin{proof}
    Let $\alpha_1\prec\alpha_2\prec\cdots\prec\alpha_{\ell}$ be a chain of $(\calA,\preceq)$. For every $i \in \{2,\ldots,\ell\}$, there exists $x_i\in\calP$ such that $x_i \in \alpha_i\setminus\alpha_{i-1}$. We claim that all $x_i$'s are distinct. Suppose for a contradiction that there exist $j<k$ such that $x_j = x_k$. Let $y \in \alpha_{k-1}$ be such that $x_j \preceq y$. Note that since $x_k$ does not belong to $\alpha_{k-1}$, we have actually $x_j \neq y$. Moreover, there exists $z \in \alpha_k$ such that $y \preceq z$. We have thus $x_k \prec z$, with both elements belonging to $\alpha_k$; a contradiction. This shows that every chain in $(\calA,\preceq)$ is of size at most $|\calP|+1$.

    We prove now by induction on $|\calP|$ that there exists a chain in $(\calA,\preceq)$ of size $|\calP|+1$. This is obviously true if $\calP = \varnothing$. Assume now that $\calP \neq \varnothing$. Let $x$ be a maximal element of $\calP$. By induction, $\calA$ possesses a chain of size $|\calP|$ such that none of its elements contains $x$. Let $\alpha$ be the maximal antichain of $(\calP,\preceq)$ in this chain of $(\calA,\preceq)$. Add $x$ to $\alpha$, and remove from it all elements $y$ such that $y \preceq x$. This leads to a new antichain $\alpha'$ such that $\alpha \prec \alpha'$, showing that there exists a chain in $(\calA,\preceq)$ of size $|\calP|+1$.
\end{proof}

\subsubsection{Two lemmas}

The notation $\tob$ and $\tor$ is extended to subsets of vertices: given a vertex $u$ and a subset $U$ of vertices, $u \tob U$ (resp.\ $U \tob u$) means the existence of at least one blue arc from $u$ to some vertex of $U$ (resp.\ from some vertex of $U$ to $u$). In case where $u \in U$, this means the existence of a blue arc to another vertex (resp.\ from another vertex) distinct from $u$. And similarly for $\tor$.

Consider a digraph $D$ as in Theorem~\ref{thm:br}. Such a digraph satisfies some properties, which we state as lemmas since they will be useful in the proof of Theorem~\ref{thm:br}.

\begin{lem}\label{lem:init}
    There is a vertex $v$ such that the implication $v \tor w \implies w \to v $ holds for all vertices $w$.
\end{lem}

\begin{proof}
    Consider the set $R$ of red arcs $(v,w)$ such that $w \not\to v$ (i.e., the set of red arcs $(v,w)$ that witness the fact the $v$ does not satisfy the implication). We claim that the restriction of $D$ to arcs in $R$ is acyclic. Suppose for a contradiction that there is a cycle, and consider such a cycle of minimal length. It cannot be of length two because this would contradict the definition of $R$. It is thus of length at least three. By \ref{item:r} and by minimality of the length, there is an arc of the cycle whose opposite arc exists in blue in the digraph. Such an arc contradicts the definition of $R$. This proves the acyclicity. The restriction to arcs in $R$ has thus a sink. Such a sink necessarily satisfies the implication.
\end{proof}

\begin{lem}\label{lem:conn}
     If there is a blue dipath from a vertex $u$ to a vertex $v$, then $u \tob v$ or $v \tor u$.
\end{lem}

\begin{proof}
    Consider a minimum-length blue dipath $P$ from a vertex $u$ to a vertex $v$. If it is of length one, then $u \tob v$. So suppose that the length of $P$ is at least two. By \ref{item:b} and by minimality of the length of $P$, there exists a red dipath from $v$ to $u$ that uses vertices of $P$ in the opposite order as that induced by $P$. Note that if $P$ is of odd length, the existence of such a path requires to use a red arc that is the opposite of an arc of $P$. Let $Q$ be such a red dipath of minimal length. By \ref{item:r} and by minimality of the length of both $P$ and $Q$, the dipath $Q$ is of length one and we have $v \tor u$.
\end{proof}

\subsubsection{The proof}

\begin{proof}[Proof of Theorem~\ref{thm:br}]
    Let $D^b$ be the digraph obtained from $D$ by keeping only the blue arcs and denote by $\calK^b$ the collection of its strongly connected components. Let $K \preceq K'$ holds for $K,K'\in \calK^b$ if there is a blue dipath from $K$ to $K'$. This makes $(\calK^b,\preceq)$ a poset. We extend the definition of $\preceq$ on the antichains of this poset as in Section~\ref{subsec:poset}. By Lemma~\ref{lem:apos}, $\preceq$ is a partial order on these antichains. By Lemma~\ref{lem:conn}, the strongly connected components of $D^b$ intersected by an independent set $I$ of $D$ form an antichain of $(\calK^b,\preceq)$, which we denote by $\alpha_I$.
  
    Let $\calI$ be the set of independent sets $I$ of $D$ such that $I \tor w \implies w \to I $ holds for all vertices $w$. By Lemma~\ref{lem:init}, $\calI$ is non-empty (and an element from $\calI$ can be determined by simply scanning the vertices of $D$).  We describe now a procedure to modify such an $I \in \calI$ when it is not a kernel, in order to get a new element $I'$ in $\calI$ such that $\alpha_I \prec \alpha_{I'}$. By finiteness, this will show the existence of a kernel. With Lemma~\ref{lem:hA}, this will even imply the polynomiality of the method.
    
    Suppose that $I \in \calI$ is not a kernel and let $U \coloneqq V(D)\setminus(I \cup N^-(I))$. This is the set of all vertices that are neither in $I$ nor absorbed by $I$. Since $I$ is not a kernel, $U$ is non-empty. According to Lemma~\ref{lem:init} applied on $D[U]$, there exists a vertex $v \in U$ such that $v \tor w \implies w \to v$ holds for all vertices $w \in U$. (Again, determining such a $v$ can be done simply by scanning the vertices of $D$.)
    
    If $I \not\to v$, then adding $v$ to $I$ leads to $I \cup \{v\} \in \calI$ such that $\alpha _I \prec \alpha_{I\cup\{v\}}$. Suppose now that $I \to v$. Define $I'$ by removing from $I$ all vertices in $N^-(v)$ and by adding $v$. The set $I'$ is independent since $v \not \to I$ by definition of $v$ and all vertices of $I$ with an arc to $v$ have been removed from $I$. Moreover, let $w$ be any vertex such that $I' \tor w$. We claim that $w \to I'$. Three cases have to be considered.

\smallskip

\begin{easylist}\ListProperties(Style1*=\scshape$\bullet$, Style2*=$\diamond\;$, Hide=2, Progressive*=3ex, Space=0.2cm, Space*=0.2cm)
    # \; {\em $w \to I \setminus N^-(v)$}. Then $w \to I'$ since $I \setminus N^-(v) \subseteq I'$ by definition.

    # \; {\em $w \to I \cap N^-(v)$}. Let $u$ be a vertex in $I \cap N^-(v)$ such that $w \to u$. We have $u \tob v$ because $v \not\to I$. If $w \tob u$, then \ref{item:b} implies $w \tob v$ because $v \not\to I$ and we are done. So, suppose that $w \tor u$, and let $v' \in I'$ such that $v' \tor w$. By \ref{item:r}, we have $w \tob v'$ (here, we use the fact that $v' \not \tor u$, which holds either because $v' \in I$ or because $v' = v$ and $v \not\to I$). Therefore, in this case, whatever is the color of the arc $(w,u)$, we have $w \to I'$.
    
    # \; {\em $w \not\to I$}. This means that $v \tor w$. Since $w$ does not belong to $I'$, it is distinct from $v$. Since $v\not\to I$, the vertex $w$ does not belong to $I \cup N^-(I)$. Therefore, the vertex $w$ belongs to $U$ and by definition of $v$ we have $w \to v$, which implies $w \to I'$.
    \end{easylist}

    \smallskip
  
Therefore, $I' \in \calI$. Moreover, no vertex of $I \cap N^-(v)$ is in the same strongly connected component of $D^b$ as $v$: indeed, the existence of a blue dipath from $v$ to $I \cap N^-(v)$ would imply by Lemma~\ref{lem:conn} $v \tob I$ or $I \tor v$, both situations contradicting the fact that $v$ belongs to $U$. This means $I' \neq I$. Since $v \in U$, every arc from $I \cap N^-(v)$ to $v$ are blue, and hence $\alpha_I \prec \alpha_{I'}$.
\end{proof}

\subsection{Discussion}
As noted in the introduction, Theorem~\ref{thm:br} is also a generalization of a theorem by Champetier~\cite{champetier1989kernels}. A {\em comparability graph} is an undirected graph admitting a transitive orientation. Champetier's theorem is a special case of the theorem of Boros and Gurvich about perfect graphs and kernel-solvability. Champetier's theorem, proved around a decade before, states: {\em Every comparability graph is kernel-$M$-solvable.} A way to see that Theorem~\ref{thm:br} implies this theorem goes as follows. Take a transitive orientation of a comparability graph. Color in red each arc if it is oriented the same way as in the transitive orientation, and in blue otherwise. This coloring satisfies the condition of Theorem~\ref{thm:br}. Actually, the original proof itself was considering such a coloring. It has inspired the proof technique we use and that has also been used elsewhere; see, e.g., the proof of kernel-$M$-solvability of perfectly orderable graphs by Blidia and Engel~\cite{blidia1992perfectly}.

Another result in the same vein as Theorem~\ref{thm:br} can be proved with similar techniques.

\begin{figure}
\centering
\includegraphics[scale=0.8]{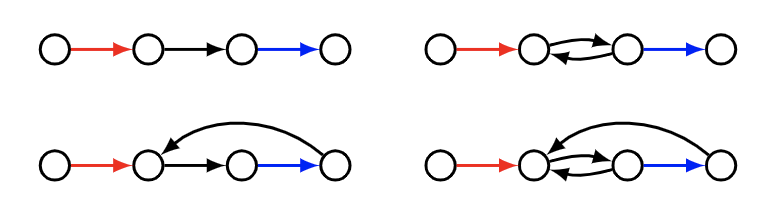}

\caption{\label{fig:4br} The induced structures forbidden by condition~\ref{cond4br} in Proposition~\ref{prop:4br}}
\end{figure}

\begin{proposition}\label{prop:4br}
    Let $D$ be a digraph whose arcs are colored in blue and red, such that the following conditions are both satisfied (see Figure~\ref{fig:4br} for the forbidden structures):
    \begin{enumerate}[label=\textup{(\roman*)}]
        \item There is no monochromatic directed cycle.\label{condmono4br}
        \item If $(v_1,v_2),(v_2,v_3),(v_3,v_4)$ is a (closed or open) directed path such that $v_1 \tor v_2$ and $v_3 \tob v_4$, then its vertices induce at least another arc not ending at $v_2$.\label{cond4br}
    \end{enumerate}
     Then $D$ has a kernel.
\end{proposition}

None of Theorem~\ref{thm:br} and Proposition~\ref{prop:4br} imply the other. Indeed, a digraph of size three such that every vertex has a blue arc to every other vertex satisfies conditions of Theorem~\ref{thm:br} but has a monochromatic directed cycle. Also, a directed blue path of length two satisfies the conditions of Proposition~\ref{prop:4br} but not those of Theorem~\ref{thm:br}.

The proof of Proposition~\ref{prop:4br} shares similarities with the proof of Theorem~\ref{thm:br}. Despite this similarity, the proof does not provide here any efficient algorithm. We  have not been able to settle the complexity of finding a kernel as in Proposition~\ref{prop:4br}. 

The proof relies on a preliminary lemma. It is a reformulation of a result by Blidia and Engel~\cite[Claim~1]{blidia1992perfectly}. We provide a proof for sake of completeness.

\begin{lem}\label{lem:acyclic}
    Consider an acyclic digraph $D$. Let $\calD$ be the digraph whose vertices are the independent sets of $D$ and where $(I,I')$ is an arc if $I \neq I'$ and $v \to I '$ for every $v \in I \setminus I'$. Then $\calD$ is acyclic.
\end{lem}

\begin{proof}
Assume for a contradiction that there is a sequence $I_1,\ldots,I_k$ of independent sets of $D$ such that $(I_i,I_{i+1})$ is an arc of $\calD$ for every $i\in [k]$ (with $k+1 = 1$). Up to a renumbering, we can assume there is a $v_1 \in I_1 \setminus I_2$. We build then an infinite sequence $v_1,v_2,\ldots$ of vertices such that for every $j$
\begin{itemize}
    \item $v_j \in I_i$ with $i = j \mod k$, 
    \item $v_j = v_{j+1}$ or $v_j \to v_{j+1}$ otherwise.
\end{itemize}
Remark that no $v_i$ belongs to $\bigcap_{j \in [k]}I_j$: this is true for $v_1$, and a direct induction shows that it is then true for all $j$. This implies that the sequence $(v_j)$ does not become constant from a certain index. By finiteness, the sequence has to travel along a cycle of $D$, but this contradicts its acyclicity.
\end{proof}

\begin{proof}[Proof of Proposition~\ref{prop:4br}]
Let $\calI$ be the set of independent sets $I\subseteq V(D)$ such that $I\tor w\implies w \to I$ holds for all vertices $w$. 
The set $\calI$ is not empty (just take a sink of the digraph obtained from $D$ by keeping only the red arcs). Consider the digraph $\calD^{\text{blue}}$ whose vertices are the independent sets in $\calI$ and where $(I,I')$  is an arc if $I \neq I'$ and $u \tob I '$ for every $u \in I \setminus I'$. According to Lemma~\ref{lem:acyclic} and since there is no blue cycle, $\calD^{\text{blue}}$ is acyclic.

We describe now a procedure to modify an $I\in \calI$ that it is not a kernel, in order to get a new element $I'\in \calI$ such that $(I,I')$ is an arc in $\calD^{\text{blue}}$. By finiteness, this will show the existence of a kernel.
    
    Suppose that $I \in \calI$ is not a kernel and let $U \coloneqq V(D)\setminus(I \cup N^-(I))$. This is the set of all vertices that are neither in $I$ nor absorbed by $I$. Since $I$ is not a kernel, $U$ is non-empty. 
    Take $v$ a sink of the digraph obtained from $D[U]$ by keeping only the red arcs.  
    If $I \not\to v$, then adding $v$ to $I$ leads to a $I'=I \cup \{v\} \in \calI$ such that $(I,I')$ is an arc of $\calD^{\text{blue}}$. Suppose now that $I \to v$. Define $I'$ by removing from $I$ all vertices in $N^-(v)$ and by adding $v$. The set $I'$ is independent since $v \not \to I$ by definition of $v$ and all vertices of $I$ with an arc to $v$ have been removed from $I$. Moreover, let $w\in V(D)$ and $v'\in I'$ be any vertex such that $v' \tor w$ (with possibly $v'=v$). We claim that $w \to I'$. Three cases have to be considered.

\smallskip

\begin{easylist}\ListProperties(Style1*=\scshape$\bullet$, Style2*=$\diamond\;$, Hide=2, Progressive*=3ex, Space=0.2cm, Space*=0.2cm)
    # \; {\em $w \to I \setminus N^-(v)$}. Then $w \to I'$ since $I \setminus N^-(v) \subseteq I'$ by definition.

    # \; {\em $w \to I \cap N^-(v)$}. Let $u$ be a vertex in $I \cap N^-(v)$ such that $w \to u$. The arc $(u,v)$ cannot be red because $v \in U$. Thus we have $u \tob v$. Condition \ref{cond4br} with $v_1=v'$, $v_2 = w$, $v_3 = u$, and $v_4 = v$ implies that there is another arc induced by these vertices, not ending at $w$. This arc cannot end at $u$: we have $v \not\to I$ and $u \in I$; we have $v' \not\to u$ because either $v'=v$ and this is already checked, or $v' \in I$ and we use the independence of $I$. Thus, the arc ends at $v$ or $v'$, and its tail is necessary $w$, which means $w \to I'$. (Note that when $v' \neq v$, we have $u \not\to v'$ again because of the independence of $I$.)
    
    # \; {\em $w \not\to I$}. This means that $v' = v$, since otherwise there would be a red arc from $I$ to $w$, and an arc from $w$ to $I$ by definition of $\calI$. Hence $v \tor w$. Since $v\not\to I$, the vertex $w$ does not belong to $I$. It does not belong to $N^-(I)$ either because this is the case we consider. Therefore, the vertex $w$ belongs to $U$ but this is a contradiction with the definition of $v$.
    \end{easylist}

    \smallskip
  
Therefore, $I' \in \calI$. 
Moreover, every vertex $u$ in $I\setminus I'$ is in $N^-(v)$ and since $v\not\to I$, we have $u\tob v$, therefore $(I,I')$ is an arc of $\calD^{\text{blue}}$. \end{proof}

We could not find any polynomial bound on the number of independent sets $I$ considered in the previous proof. As a result, we do not know whether any complexity result for finding a kernel in graphs satisfying conditions of Proposition~\ref{prop:4br} can be derived from the proof.

\section{Odd holes}\label{sec:odd-holes}

\subsection{Proof of Theorem~\ref{thm:cordes_nous}}

\subsubsection{Semi-kernels}

A {\em semi-kernel} is a subset $S$ of vertices that is independent and such that $N^+(S) \subseteq N^-(S)$. Note that a kernel is a semi-kernel, and that the empty set is also a semi-kernel. Galeana-Sánchez and Neumann-Lara~\cite{galeana-sanchez_kernel-perfect_1986} suggested the idea of using semi-kernels to systematize several proofs on the existence of kernels. The key property underlying this idea is the following.

\begin{lem}\label{lem:semik}
    If every non-empty induced subdigraph of a digraph $D$ has a non-empty semi-kernel, then $D$ has a kernel.
\end{lem}

\begin{proof}
    We prove the result by induction on the number of vertices. 
    Let $D$ be a digraph whose non-empty induced subdigraphs all have a non-empty semi-kernel. If $D$ has exactly one vertex, it has a kernel. 
    So, suppose $D$ has at least two vertices. By the assumption, $D$ itself has a non-empty semi-kernel, which we denote by $S$. If $S$ is a kernel, we are done. Otherwise, the digraph $D-N^-[S]$ is a non-empty induced subdigraph of $D$, which implies that all its induced subdigraphs have also a non-empty semi-kernel. By induction, $D-N^-[S]$ has a kernel. This kernel forms with $S$ a kernel of $D$.
\end{proof}

With this lemma, classical theorems on existence of kernels get short proofs, like Richardson's theorem~\cite{richardson_weakly_1946}; see~\cite[Chapitre 1, Section 1.1.1.1]{langlois2023kernels} for other illustrations. This lemma was also used in the original proof of the Galeana-Sánchez--Neumann-Lara theorem. Our generalization does not depart in that regard.

\subsubsection{The proof}

Our proof, as a streamlined version of the original proof of the Galeana-Sánchez--Neumann-Lara theorem, covers the new cases easily.

\begin{proof}[Proof of Theorem~\ref{thm:cordes_nous}]
The proof works by induction on the number of vertices. 

The theorem is obviously correct for the graph reduced to a single vertex. Assume now that $D$ has at least two vertices, and let $u$ be any vertex of $D$. By induction, $D \setminus N^-[u]$ has a kernel $K$, and set $K'=K \cup \{u\}$. If $K$ does not contain any vertex of $N^+(u)$, then $K'$ is a kernel of $D$. 

We can thus assume that $K$ contains at least one vertex of $N^+(u)$, and we denote by $I$ the set $K \cap N^+(u)$. Notice that $I$ is an independent set. Let $S$ be the set of vertices $v$ of $K'$ such that there exists a directed path from a vertex in $I$ to $v$ that satisfies the following two conditions:
\begin{enumerate}[label=(\roman*)]
    \item\label{alternate} the path alternates between vertices in $K'$ and vertices not in $K'$.
    \item\label{before} for all vertices $w$ and $w'$ such that $w$ comes before $w'$ on the path and such that $w'$ is not the end vertex of the path, if $w\in N^+(w')$, then $w\not\in K'$.
\end{enumerate}
Note that $S$ is non-empty and contains in particular $I$. We claim that $S$ is actually a semi-kernel. 

To show this, we first prove that $u$ is not an element of $S$. Suppose for a contradiction that $u$ belongs to $S$. Then there is a directed path $P$ from $I$ to $u$ satisfying the two conditions~\ref{alternate} and \ref{before}. Choose such a path of minimum length and denote by $v$ its origin. It is of even length because of condition~\ref{alternate}. Together with the arc $(u,v)$, the path $P$ forms a directed cycle of odd length. We consider the following two possibilities.

\smallskip

\begin{easylist}\ListProperties(Style1*=\scshape$\bullet$, Style2*=$-$, Hide=2, Indent=0.5cm, Space1=0.4cm, Space1*=0.0cm, Space2=0.2cm, Space2*=0.2cm)
# \label{heads} {\em The cycle has two chords with consecutive heads.} In particular, it has a chord whose head is in $K'$. 
The tail is not in $K$: if the head is $u$, it is by definition of $K$; if the head is not $u$, it is by independence of $K$. The tail is not $u$ either: it would contradict the minimality of $P$.
Moreover, the tail cannot come before the head on $P$
 since this would contradict the minimality of $P$. Therefore, the tail comes after the head, which is in $K$. This contradicts condition~\ref{before}.
 # \label{odd-chord} {\em The cycle has no two chords with consecutive heads.} The cycle has thus an odd chord which we denote by $(w,w')$. The vertex $w$ is distinct from $u$ since otherwise this would contradict the minimality of $P$. The vertex $w'$ is also distinct from $u$ for the same reason. Moreover, the vertices $w$, $w'$, and $u$ cannot be in this order on the directed cycle, again for the same reason of minimality of $P$. Thus, the vertices $w$ and $w'$ are such that $w$, $u$, and $w'$ are distinct and in this order on the cycle. This shows that we cannot be in the possibility offered by the second item: any odd chord is such that $u$ is on the part of the cycle from the tail of the chord to its head. Moreover, the configuration of $w$ and $w'$ with respect to $u$ and the fact that the chord is odd makes that $w$ and $w'$ are both outside $K'$. But then any short crossing chord as in Figure~\ref{fig:cordes_nous_short_odd} would either connect two elements of $K$ or connect an element of $K$ to $u$. This contradicts the definition of $K$.
\end{easylist}

\smallskip

 We show now that $S$ is a semi-kernel. 
 Since $S$ does not contain $u$, it is a subset of $K$ and is thus an independent set. 
 Consider now a vertex $s$ of $S$ with an outneighbor $w$. Note that $w$ is not in $K$ and is distinct from $u$. 
 By definition of $s$, there is a directed path $Q$ from $I$ to $s$ satisfying the two conditions~\ref{alternate} and \ref{before}. 
 If $w$ is in $N^-(u)$, then there is an arc from $w$ to a vertex of $V(Q)\cap K \subseteq S$ since otherwise 
 the path obtained by appending $(s,w)$ and $(w,u)$ to $Q$ would satisfy the two conditions~\ref{alternate} and \ref{before} and this would imply that $u$ belongs to $S$ (indeed, by independence of $K$, no arc forbidden by \ref{before} can begin at $s$). 
 If $w$ is not in $N^-(u)$, then there is an arc from $w$ to a vertex $t$ in $K$, because this latter set is a kernel of $D \setminus N^-[u]$. Either there is such a $t$ on $Q$ and then $t\in S$, or considering any such $t$ and appending $(s,w)$ and $(w,t)$ to $Q$ satisfies the two conditions~\ref{alternate} and \ref{before}, so $t\in S$. 
 In any case, there is an arc from $w$ to an element of $S$, and $S$ is therefore a semi-kernel.

We have proved that $D$ has a semi-kernel. By induction, every proper induced subdigraph of $D$ admits a kernel, and thus a semi-kernel. Lemma~\ref{lem:semik} leads to the conclusion.
\end{proof}

\subsection{Discussion}\label{subsec:discuss-odd-holes}

There is another theorem ensuring the existence of a kernel with some ``chord'' condition on the odd cycles, namely a theorem by Duchet~\cite{duchet1980graphes} stating: {\em If every odd directed cycle has two reversible arcs, then the graph has a kernel.} (A short proof based on Lemma~\ref{lem:semik} is also possible in this case.) We do not know whether there is a common generalization of this latter theorem and Theorem~\ref{thm:cordes_nous}.

Other open questions, probably more challenging, are about algorithmic versions of the Galeana-Sánchez--Neumann-Lara theorem and its generalization, Theorem~\ref{thm:cordes_nous}. For instance, we do not know whether a kernel can be computed in polynomial time under the conditions of either of these theorems. We do not even know whether the conditions of these theorems can be checked in polynomial time. 

\section{Kernel-solvability of anti-holes}\label{sec:kernel-solvability}

 The notion of $M$-clique-acyclicity (definition given in Section~\ref{subsec:terminology}) has been introduced by Meyniel according to Duchet~\cite{duchet1987parity}. It is not difficult to check that $M$-clique-acyclicity implies clique-acyclicity. Yet, by allowing reversible arcs, it covers many interesting situations not covered by simple clique-acyclicity, while being more amenable than the general clique-acyclity. Several special cases of the Berge--Duchet conjecture mentioned in Section~\ref{subsec:clique-acyclic} have indeed been first established for $M$-clique-acyclic orientations. Moreover, deciding whether a digraph is $M$-clique-acyclic can be done in polynomial time, while deciding clique-acyclicity is \NP-complete even for perfect graphs~\cite{andres2015perfect}.

Boros and Gurvich~\cite[Observation 3.13]{boros_perfects_2006} claimed that the anti-hole $\overline{C}_7$ is kernel-$M$-solvable but not kernel-solvable. This is actually not true as shown by Figure~\ref{fig:C_7}. Determining which odd anti-holes are kernel-$M$-solvable is a question that remains to be settled. More importantly, we know by the Boros--Gurvich theorem that a graph is perfect if and only if it is kernel-solvable; it is still possible that kernel-$M$-solvability already characterizes perfectness.

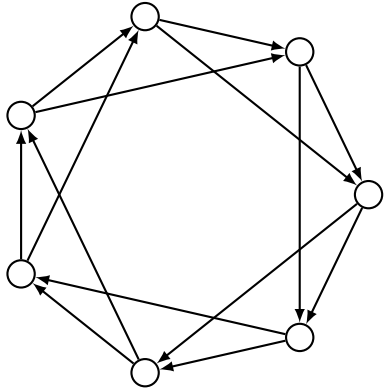
\begin{figure}
    \centering

\begin{tikzpicture} 
 [
    scale=.81,
    >=latex,
    thick,
    vertex/.style={shape=circle,draw=black},
    qk/.style={vertex,fill=black!25},
    sink/.style={vertex,fill=black!100}
  ]

\node[shape=circle,draw=black] (6) at (360/7*3:3cm){} ;
\node[shape=circle,draw=black] (3) at (360/7*1:3cm){} ;
\node[shape=circle,draw=black] (7) at (360/7*6:3cm){} ;
\node[shape=circle,draw=black] (4) at (360/7*4:3cm){} ;
\node[shape=circle,draw=black] (1) at (360/7*2:3cm){} ;
\node[shape=circle,draw=black] (5) at (360/7*7:3cm){} ;
\node[shape=circle,draw=black] (2) at (360/7*5:3cm){} ;

\path [->,>=latex](1) edge node[left] {} (3);
\path [->,>=latex](3) edge node[left] {} (5);
\path [->,>=latex](5) edge node[left] {} (7);
\path [->,>=latex](7) edge node[left] {} (2);
\path [->,>=latex](2) edge node[left] {} (4);
\path [->,>=latex](4) edge node[left] {} (6);
\path [->,>=latex](6) edge node[left] {} (1);

\path [->,>=latex](1) edge node[left] {} (5);
\path [->,>=latex](5) edge node[left] {} (2);
\path [->,>=latex](2) edge node[left] {} (6);
\path [->,>=latex](6) edge node[left] {} (3);
\path [->,>=latex](3) edge node[left] {} (7);
\path [->,>=latex](7) edge node[left] {} (4);
\path [->,>=latex](4) edge node[left] {} (1);

    \end{tikzpicture}
    \caption{\label{fig:C_7} A simple clique-acyclic orientation of $\overline{C}_{7}$ with no kernel. }
\end{figure}

\begin{proof}[Proof of Proposition~\ref{prop:odd_anti-hole}]
    Denote by $v_1,\ldots,v_n$ the vertices of the anti-hole, so that $v_iv_{i+1}$ is a non-edge for all $i$ (and with $v_{n+1} \coloneqq v_1$). Consider a simple clique-acyclic orientation of that anti-hole.
    
    We first establish that there is an $i^\star$ such that the edges $v_{i^\star-2}v_{i^\star}$ and $v_{i^\star}v_{i^\star+2}$ are both oriented towards $v_{i^\star}$. Suppose for a contradiction that for every $i$, the arc $(v_i,v_{i-2})$ or the arc $(v_i,v_{i+2})$ exists. Without loss of generality, we can assume that the arc $(v_1,v_3)$ exists. Then $(v_3,v_5)$, $(v_5,v_7)$, ..., $(v_n,v_2)$, $(v_2,v_4)$, etc. exist as well, i.e., $(v_i,v_{i+2})$ exists for all $i$. (Here, we use that $n$ is odd.) Since the orientation is clique-acyclic, we get that $(v_i,v_{i+4})$ exists for all $i$. Repeating this argument, we get that an edge $v_iv_j$ with $i<j$ and $2\leq j-i \leq n-2$ is oriented $(v_i,v_j)$ precisely when $j-i$ is even. Since $n \geq 9$, the arc $(v_1,v_7)$ exists in $D$. The arcs $(v_7,v_4)$ and $(v_4,v_1)$ also exist; this contradicts the simple clique-acyclic orientation of $D$, and shows the existence of the desired $i^\star$.

    We know that $D-N^-[v_{i^\star}]$ admits a kernel $K$, because $\overline C_n$ is minimally imperfect and the Boros--Gurvich theorem applies. If $K$ does not intersect $N^+(v_{i^\star})$, then $K \cup \{v_{i^\star}\}$ is a kernel of $D$. So, assume that $K$ intersects $N^+(v_{i^\star})$, and let $x$ be a vertex in this intersection. We prove now that there exists a semi-kernel in $D$. Since every strict subdigraph of $D$ has a kernel by the Boros--Gurvich theorem, Lemma~\ref{lem:semik} will lead to the desired conclusion. If $x$ is a sink of $D$, then $x$ is a semi-kernel, and we are done. 
    
    We can thus assume that $x$ is not a sink of $D$. Pick an outneighbor $y$ of $x$. Note that $x$ has no outneighbor in $N^-[v_{i^\star}]$ by the simple clique-acyclicity of $D$. In particular, $y$ does not belong to $N^-[v_{i^\star}]$. Hence, there is vertex $z$ in $K$ such that the arc $(y,z)$ exists. We have $K=\{x,z\}$ because the independence number of $D - N^-[v_{i^\star}]$ is at most two, $D$ being an anti-hole. Suppose for a contradiction that $K$ is not a semi-kernel of $D$. Hence, there exists a vertex $t$ in $N^+(K) \setminus N^-(K)$. The vertex $t$ is in $N^-[v_{i^\star}]$ because $K$ is kernel of $D-N^-[v_{i^\star}]$, and it is not in $N^+(x)$ as $x$ has no outneighbor in $N^-[v_{i^\star}]$. Thus, $t$ is in $N^+(z)$. Notice that $z$ is neither $v_{i^\star-1}$, nor $v_{i^\star+1}$ since otherwise, this would imply that $x$ is $v_{i^\star}$ or an inneighbor of $v_{i^\star}$ because $\{x,z\}$ is independent. Therefore, there is an arc $(v_{i^\star},z)$, which contradicts, together with the existence of the arc $(z,t)$, either the simplicity of the orientation (in case $t=v_{i^\star}$), or the clique-acyclicity because the arc $(t,v_{i^\star})$ exists (in case $t \neq v_{i^\star}$).
\end{proof}

\bibliographystyle{amsplain}

\bibliography{quasi-kernel}

\end{document}